\documentclass{amsart}

\theoremstyle{definition}
\newtheorem{theorem}{Theorem}[section]

\newtheorem{lemma}[theorem]{Lemma}

\theoremstyle{definition}

\theoremstyle{remark}
\newtheorem{remark}[theorem]{Remark}

\newcommand{\norm}[1]{\left\lVert#1\right\rVert}
\newcommand{\ap}{\left( A^{(1)}, A^{(2)}, \cdots ,A^{(m)} \right)}
\newcommand{\mmp}{\left( M^{(1)}, M^{(2)}, \cdots ,M^{(m)} \right)}
\newcommand{\bp}{\left( B^{(1)}, B^{(2)}, \cdots ,B^{(m)} \right)}

\newcommand{\ep}{\left( E^{(1)}, E^{(2)}, \cdots ,E^{(m)} \right)}

\allowdisplaybreaks[1]

\numberwithin{equation}{section}

\begin{document}

\title[Perturbation Bounds for Joint Spectrum of Tuple of Matrices]{Relative Perturbation Bounds for the Joint Spectrum of Commuting Tuple of Matrices}


\author{Arnab Patra}
\address{Department of mathematics, Indian Institute of Technology Kharagpur, India 721302}
\curraddr{}
\email{arnptr91@gmail.com}
\thanks{}

\author{P. D. Srivastava}
\address{Department of mathematics, Indian Institute of Technology Kharagpur, India 721302}
\curraddr{}
\email{pds@maths.iitkgp.ernet.in}
\thanks{}

\subjclass[2010]{Primary 15A42; Secondary 15A18}

\keywords{Joint eigenvalues, Clifford algebra, commuting tuple of matrices.}

\date{}

\dedicatory{}

\begin{abstract} 
In this paper, we study the relative perturbation bounds for joint eigenvalues of commuting tuples of normal $n \times n$ matrices. Some Hoffman-Wielandt type relative perturbation bounds are proved using the Clifford algebra technique. A result is also extended for diagonalizable matrices which improves a relative perturbation bound for single matrices.
\end{abstract}

\maketitle
\section{\bf Introduction}

Let $A=\left( A^{(1)}, A^{(2)}, \cdots ,A^{(m)} \right)$ be an $m$-tuple commuting $n \times n$ matrices acting on $\mathbb{C}^n.$ A joint eigenvalue of $A$ is an element $\lambda = \left( \lambda^{(1)}, \lambda^{(2)}, \cdots ,\lambda^{(m)} \right) \in \mathbb{C}^m$ such that,
\begin{equation*}
A^{(j)}x = \lambda^{(j)}x \ \ \mbox{for} \ \ j=1,2, \cdots, m
\end{equation*}
holds for some non-zero vector $x \in \mathbb{C}^n.$ The vector $x$ is called joint eigenvector. The set of all joint eigenvalues of $A$ is called the joint spectrum of $A$.

The main concern of perturbation theory of matrix eigenvalues is to estimate the error when the eigenvalues of a matrix are approximated by the eigenvalues of a perturbed matrix. Let $A$ and $B$ are two $n \times n$ matrix and $\{\lambda_1, \lambda_2, \cdots, \lambda_n\}$ and $\{\mu_1, \mu_2, \cdots, \mu_n \}$ are their eigenvalues respectively. An important result in the direction of absolute type perturbation bound is Hoffman-Weildant Theorem \cite{hoffman}. Which states that, if $A$ and $B$ are normal matrices, then there exist a permutation $\pi$ of $\{1,2, \cdots, n\}$ such that,
\begin{equation*}
\left( \sum\limits_{i=1}^{n} \Big| \lambda_i - \mu_{\pi(i)} \Big|^2 \right)^{\frac{1}{2}} \leq \|A-B\|_F
\end{equation*}
where $\|.\|_F$ denotes the Frobenius Norm. Many researchers generalized this result in several directions.  In 1993 Bhatia and Bhattacharyya \cite{bhatia} generalized the Hoffman-Weildant Theorem for joint eigenvalues of $m$-tuple of commuting normal matrices. More results on absolute type perturbation bounds for joint eigenvalues can be found in \cite{freedman, pryde1}.

In 1998, relative perturbation bounds for eigenvalues of diagonalizable matrices were studied by Eisenstat and Ipsen \cite{siam}. It was proved in \cite{siam} that, if $A$ and $B$ both are diagonalizable and $A$ is non-singular, then there exists a permutation $\pi$ of $\{1,2, \cdots, n\}$ such that,
\begin{equation*}
\left( \sum\limits_{i=1}^{n} \Big| \frac{\lambda_i - \mu_{\pi(i)}}{\lambda_i}  \Big|^2 \right)^{\frac{1}{2}} \leq \kappa(X) \kappa(\tilde{X}) \|A^{-1}(A-B)\|_F
\end{equation*}
where $X$ and $\tilde{X}$ are the invertible matrices which diagonalize $A$ and $B$ respectively and $\kappa(X)= \|X\| \|X^{-1}\|$ be the condition number of the matrix $X.$ Furthermore, Li and Sun \cite{li4} obtained bounds for a normal non-singular matrix and a arbitrary matrix and Li and Chen \cite{li3} generalized it for diagonalizable matrices.

In this paper, we have proved some relative perturbation bounds for joint eigenvalues of $m$-tuple of commuting normal and diagonalizable matrices using the Clifford algebra technique proposed by McIntosh and Pryde \cite{pryde}

\section{\bf The Clifford Algebra Technique}
For the convenience of the reader we briefly discuss the Clifford algebra technique. Let $\mathbb{R}^m$ be the real vector space of dimension $m$ and let $e_1, e_2, \cdots, e_m$ be the basis. The Clifford algebra is an algebra generated by $e_1, e_2, \cdots, e_m$ with the following relations
\begin{equation*}
e_i e_j = -e_j e_i \ \ \mbox{for} \ \ i \neq j \ \ \mbox{and} \ \ e_i^2 = -1 \ \ \mbox{for all} \ \ i 
\end{equation*} 
and it is denoted by $\mathbb{R}_{(m)}.$ Then $\mathbb{R}_{(m)}$ is an algebra over $\mathbb{R}$ of dimension $2^m.$ Let $S$ be a subset of $\{1, \cdots, m\}$ such that $S= \{s_1, s_2, \cdots, s_k\}$ with $1 \leq s_1 < s_2 < \cdots < s_k \leq m. $ Then the elements $e_S = e_{s_1} e_{s_2} \cdots e_{s_k}$ form a basis of $\mathbb{R}_{(m)}$ where $S$ runs over the all subsets of $\{1, \cdots, m\}$ with $e_{\phi} = 1.$ An element $\alpha$ of $\mathbb{R}_{(m)}$ is of the form $\alpha = \sum\limits_{S} \alpha_S e_S $ where $\alpha_S \in \mathbb{R}.$ If $\beta = \sum\limits_{S} \beta_S e_S ,$ $\beta_S \in \mathbb{R}$ be another element of $\mathbb{R}_{(m)},$ the inner product of $\alpha$ and $\beta$ is defined as,
\begin{equation*}
\left\langle \alpha, \beta \right\rangle = \sum\limits_{S} \alpha_S \beta_S.
\end{equation*} 
Under this inner product $\mathbb{R}_{(m)}$ becomes an Hilbert space with the orthonormal basis $e_S.$ Now the tensor product $\mathbb{C}^n \otimes \mathbb{R}_{(m)}$ where
\begin{equation*}
\mathbb{C}^n \otimes \mathbb{R}_{(m)} = \left\lbrace \sum\limits_{S} x_S \otimes e_S : x_S \in \mathbb{C}^n \right\rbrace,
\end{equation*}
 is a Hilbert space under the inner product
\begin{equation*}
\left\langle x,y \right\rangle = \left\langle \sum\limits_{S} x_S \otimes e_S, \sum\limits_{S} y_S \otimes e_S \right\rangle = \sum\limits_{S} \left\langle x_S,y_S \right\rangle
\end{equation*}
where $x_S,y_S \in \mathbb{C}^n$ and the inner product in the right hand side be the usual inner product in $\mathbb{C}^n.$ Therefore  the norm on $\mathbb{C}^n \otimes \mathbb{R}_{(m)}$ is defined by,
\begin{equation*}
\norm{ \sum\limits_{S} x_S \otimes e_S } = \left( \sum\limits_{S} \|x_S\|^2 \right)^{\frac{1}{2}}
\end{equation*}
where the norm in the right hand side is the usual norm in $\mathbb{C}^n.$ Let $M_n$ be the space of $n \times n$ matrices of complex entries. Then $M_n \otimes \mathbb{R}_{(m)}$ is a linear space and for $A \in M_n \otimes \mathbb{R}_{(m)},$ $A$ is of the form $A= \sum\limits_{S}A_S \otimes e_S,$ where $A_S \in M_n.$ Each element $A= \sum\limits_{S}A_S \otimes e_S \in M_n \otimes \mathbb{R}_{(m)}$ acts on the elements $x =\sum\limits_{T} x_T \otimes e_T \in \mathbb{C}^n \otimes \mathbb{R}_{(m)}$ by

\begin{equation*}
Ax= \left( \sum\limits_{S}A_S \otimes e_S \right) \left(  \sum\limits_{T} x_T \otimes e_T \right) =  \sum\limits_{S,T}A_S x_T \otimes e_S e_T.
\end{equation*}

For an $m$-tuple $A=\left( A^{(1)}, A^{(2)}, \cdots ,A^{(m)} \right)$ of $n \times n$ complex matrices, the corresponding Clifford operator $Cliff(A) \in M_n \otimes \mathbb{R}_{(m)}$ is defined by,
\begin{equation*}
Cliff(A) = i \sum\limits_{j=1}^{m} A^{(j)} \otimes e_j.
\end{equation*}

\section{\bf Relative perturbation bounds}
Throughout this paper, $S_n$ be the set of all $n!$ permutations of $\{1,2, \cdots, n\}$ and $\|.\|_F,$ $\norm{.}$ denote the Frobenius norm and the usual operator norm respectively. $\Re(z)$ denotes the real part of a complex number $z.$ A square matrix of non-negative real numbers is called a \textit{doubly stochastic} matrix if each row and each column sum is 1. A \textit{permutation matrix} is a square matrix such that every row and column contains exactly one entry 1 and 0 everywhere else. 

 \begin{lemma} (\cite[Lemma 1]{bhatia})
 Let $A= \ap$ be any $m$-tuple of operators in $\mathbb{C}^n$ and let $Cliff(A)$ be the corresponding Clifford operator. Then
 \begin{equation*}
 \|Cliff(A)\|_F^2 = 2^m \sum\limits_{k=1}^m \|A^{(k)}\|_F^2.
 \end{equation*}
 \end{lemma}

 \begin{lemma} (\cite{bhatia})
 If $P$ is any operator of $\mathbb{C}^n$. Then 
 \begin{enumerate}
 \item[(i)] $trace(P \otimes e_T) = 0$ for any non empty subset $T$ of $\{1,2, \cdots, m\}$,
 \item[(ii)] $trace(P \otimes e_\phi) = 2^m~trace{P}.$
 \end{enumerate}
\end{lemma}

 Let $A=\ap$ and $B=\bp$ are two $m$-tuple of normal $n \times n$ matrices such that $B=A+E,$ where $E = \ep$ is the perturbation given to $A.$ Also let $\alpha_i = \left( \alpha_i^{(1)}, \alpha_i^{(2)}, \cdots ,\alpha_i^{(m)} \right)$  and $\beta_i = \left( \beta_i^{(1)}, \beta_i^{(2)}, \cdots ,\beta_i^{(m)} \right)$ are the joint eigenvalues of $A$ and $B$ respectively.

\begin{theorem}\label{th2.1}
If $A$ and $B= A+E$ are $m$-tuple of normal matrices as mentioned above and for $k=1,2, \cdots m,$ each $A^{(k)}$ is non-singular, then there exists a permutation $\pi$ of $S_n$ such that
\begin{equation*}
\sum\limits_{j=1}^n \sum\limits_{k=1}^m \left| \frac{\alpha_j^{(k)} - \beta_{\pi(j)}^{(k)}}{\alpha_j^{(k)}} \right|^2 \leq \sum\limits_{k=1}^{m} \norm{{A^{(k)}}^{-1}E^{(k)}}_F^2.
\end{equation*}
\end{theorem} 
 
\begin{proof} Since $E=B-A$ therefore
\begin{equation}\label{eq2.002}
 E^{(k)} = B^{(k)} - A^{(k)} \Rightarrow {A^{(k)}}^{-1} E^{(k)} =  {A^{(k)}}^{-1} B^{(k)} - I, 
\end{equation}
for $k=1,2, \cdots m.$

Let $C = \left({A^{(1)}}^{-1} B^{(1)}, \cdots , {A^{(m)}}^{-1} B^{(m)} \right)$ , $D = \left( {A^{(1)}}^{-1} E^{(1)}, \cdots , {A^{(m)}}^{-1} E^{(m)} \right).$ Also let $\tilde{I} = (I,I, \cdots,I)$ be the $m$-tuple of identity matrices of order $n \times n.$ Then from (\ref{eq2.002}) we have
\begin{equation} \label{eq2.003}
C - \tilde{I} = D.
\end{equation}
We can choose orthonormal bases $\{u_1, u_2,\cdots, u_n\}$ and $\{v_1,v_2, \cdots, v_n\}$ of $\mathbb{C}^n$ such that,
\begin{equation*}
A^{(k)}u_j = \alpha_j^{(k)}u_j, \ \ \ \ \  B^{(k)}v_j = \beta_j^{(k)}v_j, \ \ \ \mbox{for} \ \  j=1,2, \cdots, n \ \ \mbox{and} \ \  k=1,2, \cdots, m.
\end{equation*}
Let $P_j$ and $Q_j$ denotes the orthogonal projection operator to spaces spanned by $u_j$ and $v_j$ respectively. Then for $k=1,2, \cdots, m,$
\begin{equation*}
A^{(k)}=\sum\limits_{j=1}^{n} \alpha_j^{(k)}P_j, \ \ \
B^{(k)}=\sum\limits_{l=1}^{n} \beta_l^{(k)}Q_l.
\end{equation*}
Now,
\begin{eqnarray*}
Cliff(C)&=& i\sum\limits_{k=1}^m {A^{(k)}}^{-1}B^{(k)} \otimes e_k\\
         &=& i\sum\limits_{k=1}^m \left(\sum\limits_{j=1}^n \sum\limits_{l=1}^n {\alpha_j^{(k)}}^{-1}P_j \beta_l^{(k)}Q_l \right) \otimes e_k\\
         &=& i \sum\limits_{j,l=1}^n \left(\sum\limits_{k=1}^m {\alpha_j^{(k)}}^{-1} \beta_l^{(k)}I \otimes e_k \right) (P_jQ_l \otimes e_{\phi}). 
\end{eqnarray*}
Similarly
\begin{equation*}
Cliff(\tilde{I})= i \sum\limits_{r=1}^n \left(\sum\limits_{t=1}^m  I \otimes e_t \right) (Q_r \otimes e_{\phi}).
\end{equation*}

\begin{eqnarray*}\label{ar1}
\mbox{Now} \quad & & trace  \left[Cliff(\tilde{I})  Cliff(C)^*  \right]\nonumber \\
 &=& - trace \left[ \sum\limits_{j,l,r=1}^n \left(\sum\limits_{t=1}^m I \otimes e_t \right) \left( \sum\limits_{k=1}^m  \overline{{\alpha_j^{(k)}}^{-1}\beta_l^{(k)}} I \otimes e_k  \right) (Q_r \otimes e_{\phi})  (Q_l P_j  \otimes e_{\phi}) \right]\nonumber \\
 &=& - trace \left[ \sum\limits_{j,l,r=1}^n \left(\sum\limits_{t=1}^m I \otimes e_t \right) \left( \sum\limits_{k=1}^m  \overline{{\alpha_j^{(k)}}^{-1}\beta_l^{(k)}} I \otimes e_k  \right)  (Q_rQ_l P_j \otimes e_{\phi}) \right]\nonumber \\
 &=& - trace \sum\limits_{j,l,r=1}^n  \left[ -\left( \sum\limits_{k=1}^m \overline{{\alpha_j^{(k)}}^{-1}\beta_l^{(k)}} \right) (Q_r Q_lP_j \otimes e_{\phi})\right]\nonumber\\
 && - trace \sum\limits_{j,l,r=1}^n \left[\sum\limits_{\substack{k,t=1\\k \neq t}}^m \left( \overline{{\alpha_j^{(k)}}^{-1}\beta_l^{(k)}} -  \overline{{\alpha_j^{(t)}}^{-1}\beta_l^{(t)}} \right)(Q_r Q_lP_j \otimes e_t e_k)  \right]\nonumber\\
 &=&  \sum\limits_{j,l,r=1}^n \sum\limits_{k=1}^m \overline{{\alpha_j^{(k)}}^{-1}\beta_l^{(k)}}   trace (Q_r Q_lP_j \otimes e_{\phi})\\
 &=& 2^m \sum\limits_{j,l=1}^n \sum\limits_{k=1}^m \overline{{\alpha_j^{(k)}}^{-1}\beta_l^{(k)}}  trace ( Q_lP_j) = 2^m \sum\limits_{j,l=1}^n \sum\limits_{k=1}^m \overline{{\alpha_j^{(k)}}^{-1}\beta_l^{(k)}}  trace ( P_jQ_l).
\end{eqnarray*}

\begin{eqnarray*}\label{ar2}
\mbox{Also,} ~~ \|Cliff(C)\|_F^2 &=& 2^m \sum\limits_{k=1}^m \|{A^{(k)}}^{-1}B^{(k)}\|_F^2 \nonumber\\
				  &=& 2^m \sum\limits_{k=1}^m \norm{\sum\limits_{j=1}^{n} {\alpha_j^{(k)}}^{-1}P_j \sum\limits_{l=1}^{n} \beta_l^{(k)}Q_l}_F^2 \nonumber\\
				  &=& 2^m \sum\limits_{k=1}^m \norm{\sum\limits_{j,l=1}^{n} {\alpha_j^{(k)}}^{-1} \beta_l^{(k)} P_j Q_l}_F^2 \nonumber\\
				  &=& 2^m \sum\limits_{k=1}^m trace \left[ \left( \sum\limits_{j,l=1}^{n} {\alpha_j^{(k)}}^{-1} \beta_l^{(k)} P_j Q_l \right) \left( \sum\limits_{r,t=1}^{n} {\alpha_r^{(k)}}^{-1} \beta_t^{(k)} P_r Q_t \right)^* \right] \nonumber\\
				 &=& 2^m \sum\limits_{k=1}^m trace \left[  \sum\limits_{j,l,r=1}^{n} {\alpha_j^{(k)}}^{-1} \beta_l^{(k)} \overline{{\alpha_r^{(k)}}^{-1} \beta_l^{(k)}} P_j Q_lP_r \right]\nonumber\\
				&=& 2^m \sum\limits_{k=1}^m  \sum\limits_{j,l,r=1}^{n} {\alpha_j^{(k)}}^{-1} \beta_l^{(k)} \overline{{\alpha_r^{(k)}}^{-1} \beta_l^{(k)}} trace( P_j Q_lP_r) \nonumber\\
				 &=& 2^m \sum\limits_{k=1}^m \sum\limits_{j,l=1}^{n} \Big|{\alpha_j^{(k)}}^{-1} \beta_l^{(k)}  \Big|^2 trace (P_j Q_l).
\end{eqnarray*}

Let $W=(w_{ij})$ where $w_{ij}= trace( P_i Q_j).$ It can be easily verified that, $W$ is a doubly stochastic matrix. Hence by well-known Birkhoff's Theorem $W$ is convex combination of permutation matrices. Therefore
\begin{equation*}
W= \sum\limits_{s=1}^{n!} t_s P_s,~~ t_s \geq 0,~~ \sum\limits_{s=1}^{n!} t_s = 1
\end{equation*}
where $P_s$ are the permutation matrices and let $\pi_s$ be the corresponding permutation. Finally from (\ref{eq2.003}) we have,
\begin{eqnarray*}
\|Cliff(D)\|_F^2 &=& \|Cliff(C - \tilde{I})\|_F^2\\
				 &=& \|Cliff(C)\|_F^2 + \|Cliff(\tilde{I})\|_F^2 - 2 \Re \left(trace ~~( Cliff(C)^* Cliff(\tilde{I}))\right)\\
				 &=& 2^m \sum\limits_{s=1}^{n!} t_s \sum\limits_{k=1}^m \sum\limits_{j=1}^n \left[1 +  \Big|{\alpha_j^{(k)}}^{-1} \beta_{\pi_s(j)}^{(k)}  \Big|^2 - 2 \Re \left(\overline{{\alpha_j^{(k)}}^{-1} \beta_{\pi_s(j)}^{(k)} }\right)  \right]\\
				 &\geq & 2^m  \min\limits_{s} \sum\limits_{k=1}^m \sum\limits_{j=1}^n \left[1 +  \Big|{\alpha_j^{(k)}}^{-1} \beta_{\pi_s(j)}^{(k)}  \Big|^2 - 2 \Re \left(\overline{{\alpha_j^{(k)}}^{-1} \beta_{\pi_s(j)}^{(k)} }\right)  \right] \\
				 &=& 2^m \sum\limits_{j=1}^n \sum\limits_{k=1}^m \left| \frac{\alpha_j^{(k)} - \beta_{\pi(j)}^{(k)}}{\alpha_j^{(k)}} \right|^2.
\end{eqnarray*}
Hence the result is proved.
\end{proof}

\begin{remark}
Sun \cite{sun} has generalized the Hoffman-Wielandt inequality for the case when one matrix is normal and other is arbitrary. Similarly, Theorem \ref{th2.1} can be extended when one tuple of matrices are arbitrary. Let $A=\ap$ and $B=\bp$ are two $m$-tuples of commuting matrices in $M_n$ with joint eigenvalues $\alpha_i = \left( \alpha_i^{(1)}, \alpha_i^{(2)}, \cdots ,\alpha_i^{(m)} \right)$  and $\beta_i = \left( \beta_i^{(1)}, \beta_i^{(2)}, \cdots ,\beta_i^{(m)} \right)$ such that each $A^{(k)}$ is normal and non-singular. Since $B^{(1)}, B^{(2)}, \cdots, B^{(m)}$ are commuting so they can be reduced to upper triangular form by a single unitary matrix. Then Using Theorem \ref{th2.1} and the proof of Theorem 1.1 of \cite{sun} it can be established that, if $A$, $B(=A+E)$ are two $m$-tuple of commuting matrices as mentioned above, then there exists a permutation $\sigma$ in $S_n$ such that
\begin{equation*}
\sum\limits_{j=1}^n \sum\limits_{k=1}^m \left| \frac{\alpha_j^{(k)} - \beta_{\sigma(j)}^{(k)}}{\alpha_j^{(k)}} \right|^2 \leq n \sum\limits_{k=1}^{m} \norm{{A^{(k)}}^{-1}}^2 \norm{E^{(k)}}_F^2.
\end{equation*}
When we relax the normality condition on each $B^{(k)},$ the constant $n$ appears on the right hand side of the above inequality. This constant $n$ is best possible. We can verify it by considering the following $n \times n$ matrices
\[A^{(k)}=
\begin{bmatrix}
0      & k      & 0      & \cdots & 0     \\
0      & 0      & k      & \cdots & 0      \\
\vdots & \vdots & \vdots &  & \vdots \\
0      & 0      & 0      & \cdots & k      \\
k      & 0      & 0      & \cdots & 0      
\end{bmatrix}, \quad
B^{(k)}=
\begin{bmatrix}
0      & k      & 0      & \cdots & 0     \\
0      & 0      & k      & \cdots & 0      \\
\vdots & \vdots & \vdots &  & \vdots \\
0      & 0      & 0      & \cdots & k      \\
0      & 0      & 0      & \cdots & 0      
\end{bmatrix} 
\]
where $k$ runs over $1,2, \cdots, m.$
\end{remark}

Now we prove the following diagonalizable analogue of Theorem \ref{th2.1}.

\begin{theorem}\label{th2.2}
If $A$ and $B$ are $m$-tuple of diagonalizable matrices and for $k=1,2, \cdots m,$ each $A^{(k)}$ is non-singular, then there exists a permutation $\pi$ of $S_n$ such that
\begin{equation*}
\sum\limits_{j=1}^n \sum\limits_{k=1}^m \left| \frac{\alpha_j^{(k)} - \beta_{\pi(j)}^{(k)}}{\alpha_j^{(k)}} \right|^2 \leq \kappa(P)^2 \kappa(Q)^2 \sum\limits_{k=1}^{m} \norm{{A^{(k)}}^{-1} \left( B^{(k)} - A^{(k)} \right)}_F^2,
\end{equation*}
\end{theorem}
where $\kappa(P) = \norm{P} \norm{P^{-1}}$ be the condition number of $P.$ To prove this theorem, first we need to prove the following lemma which is slightly different from the result proved in \cite[p. 216]{sun2}.

\begin{lemma} \label{lem1}
If $M$ and $N$ are normal matrices and $\Sigma = diag(\sigma_1, \sigma_2, \cdots, \sigma_n )$ with $\sigma_1 \geq \sigma_2 \geq \cdots \geq \sigma_n \geq 0$ then 
\begin{equation*}
\norm{M\Sigma N - \Sigma}_F \geq \sigma_n \norm{MN - I}_F.
\end{equation*}
\end{lemma}

\begin{proof}
To prove this, set $\Omega = \Sigma - \sigma_n I.$ Clearly the diagonal elements of $\Omega$ are non-negative. Now
\begin{eqnarray*}
&&\norm{M\Sigma N - \Sigma}_F^2 - \sigma_n^2 \norm{MN - I}_F^2\\ && = \norm{M (\Omega + \sigma_n I) N - (\Omega + \sigma_n I)}_F^2 - \sigma_n^2 \norm{MN - I}_F^2\\
&& = \norm{(M \Omega N - \Omega) + \sigma_n(MN - I) }_F^2 - \sigma_n^2 \norm{MN - I}_F^2\\
&& = \norm{(M \Omega N - \Omega)}_F^2 + 2 \sigma_n \Re \left\lbrace trace [(M \Omega N - \Omega)^* (MN - I)] \right\rbrace \\
&& = \norm{(M \Omega N - \Omega)}_F^2 + \sigma_n trace \left\lbrace \Omega [(MN - I)^* (MN - I) + (MN - I) (MN - I)^*] \right\rbrace\\
 & & \geq  0.
\end{eqnarray*}
\end{proof}

\textit{Proof of Theorem \ref{th2.2}}: \quad
Given $A=\ap$ and $B=\bp$ are two $m$-tuple of commuting diagonalizable matrices. Then $\alpha_i = \left( \alpha_i^{(1)}, \alpha_i^{(2)}, \cdots ,\alpha_i^{(m)} \right)$  and $\beta_i = \left( \beta_i^{(1)}, \beta_i^{(2)}, \cdots ,\beta_i^{(m)} \right)$ are joint eigenvalues of $A$ and $B$ respectively since there are two non-singular matrices $P$ and $Q$ such that for $k= 1,2, \cdots, m$ 
\begin{eqnarray*}
P A^{(k)} P^{-1} &=& D_1^{(k)} = diag \left(\alpha_1^{(k)}, \alpha_2^{(k)}, \cdots ,\alpha_n^{(k)} \right),\\
Q B^{(k)} Q^{-1} &=& D_2^{(k)} = diag \left(\beta_1^{(k)}, \beta_2^{(k)}, \cdots ,\beta_n^{(k)} \right).
\end{eqnarray*}

 \begin{eqnarray*}
\mbox{Now}\quad \norm{{A^{(k)}}^{-1}\left(B^{(k)} - A^{(k)} \right)}_F^2 &=& \norm{{A^{(k)}}^{-1}B^{(k)} - I}_F^2 \\ &=& \norm{P^{-1} {D_1^{(k)}}^{-1} P Q^{-1} D_2^{(k)} Q - I}_F^2\\
&\geq & \norm{P}^{-2} \norm{Q^{-1}}^{-2} \norm{{D_1^{(k)}}^{-1} P Q^{-1} D_2^{(k)}  - P Q^{-1}}_F^2.
\end{eqnarray*}
Let $U\Sigma V^*$ be the singular value decomposition of $P Q^{-1}$ and $\sigma_n$ be the smallest diagonal element of $\Sigma,$ then from above relation we have
\begin{eqnarray*}
\norm{{A^{(k)}}^{-1}\left(B^{(k)} - A^{(k)} \right)}_F^2 &\geq& \norm{P}^{-2} \norm{Q^{-1}}^{-2} \norm{{D_1^{(k)}}^{-1} U\Sigma V^* D_2^{(k)}  - U\Sigma V^*}_F^2\\
&\geq & \norm{P}^{-2} \norm{Q^{-1}}^{-2} \norm{ \left( U^*{D_1^{(k)}}^{-1} U \right) \Sigma \left( V^* D_2^{(k)} V \right)  - \Sigma}_F^2\\
&=& \norm{P}^{-2} \norm{Q^{-1}}^{-2} \norm{ {M^{(k)}}^{-1}  \Sigma N^{(k)}  - \Sigma}_F^2
\end{eqnarray*}
where $M^{(k)} = U^*{D_1^{(k)}} U$ and $N^{(k)} = V^* D_2^{(k)} V$ and for each $k$, $M^{(k)}$ and $N^{(k)}$ are normal. Then from Lemma \ref{lem1}
\begin{equation*}
\norm{{A^{(k)}}^{-1}\left(B^{(k)} - A^{(k)} \right)}_F^2 \geq \sigma_n^2 \norm{P}^{-2} \norm{Q^{-1}}^{-2} \norm{ {M^{(k)}}^{-1} N^{(k)}  - I}_F^2.
\end{equation*}
Finally we have
\begin{equation*}
\norm{ {M^{(k)}}^{-1} \left( N^{(k)}  - M^{(k)} \right)}_F^2 \leq \kappa(P)^2 \kappa(Q)^2 \norm{{A^{(k)}}^{-1}\left(B^{(k)} - A^{(k)} \right)}_F^2.
\end{equation*}
Now applying Theorem \ref{th2.1} for the commuting tuples of normal matrices $M= \mmp$ and $N = \left( N^{(1)}, N^{(2)}, \cdots ,N^{(m)} \right)$ from above relation we have the required result.

\begin{remark}
Corollary 5.2 of \cite{siam} is a special case of Theorem \ref{th2.2} for $m=1.$
\end{remark}

\end{document}